\documentclass[11pt]{article}

\usepackage{amssymb}
\usepackage{amsthm}
\usepackage{mathtools}
\usepackage{mathrsfs}
\usepackage{stmaryrd}
\usepackage{xspace}
\usepackage{xcolor}
\usepackage{tikz}

\usepackage[
  letterpaper,
  left=1in,
  right=1in,
  top=1in,
  bottom=1in
]{geometry}

\usepackage[hidelinks]{hyperref}
\usepackage{cleveref}

\hypersetup{
  pdftitle={Real-analytic finitely forcible kernels},
  pdfauthor={Junchi Zhang}
}

\newcommand{\bemph}[1]{{\normalfont#1}}
\newcommand{\ep}[1]{\bemph{(}#1\bemph{)}}

\newtheoremstyle{bfnote}%
{}{}%
{\slshape}{}%
{\bfseries}{\bfseries.}%
{ }%
{\thmname{#1}\thmnumber{ #2}\thmnote{ \ep{\normalfont{}#3}}}

\theoremstyle{bfnote}
\newtheorem{thm}{Theorem}[section]
\newtheorem*{thm*}{Theorem}
\newtheorem{prop}[thm]{Proposition}
\newtheorem{lem}[thm]{Lemma}

\newtheorem*{coro*}{Corollary}

\theoremstyle{definition}

\newtheorem*{dfn*}{Definition}

\newtheorem*{exmp*}{Example}

\theoremstyle{remark}
\newtheorem*{ques*}{Question}
\newtheorem*{rem*}{Remark}

\crefname{thm}{theorem}{theorems}
\Crefname{thm}{Theorem}{Theorems}
\crefname{lem}{lemma}{lemmas}
\Crefname{lem}{Lemma}{Lemmas}
\crefname{clm}{claim}{claims}
\Crefname{clm}{Claim}{Claims}
\crefname{prop}{proposition}{propositions}
\Crefname{prop}{Proposition}{Propositions}
\crefname{coro}{corollary}{corollaries}
\Crefname{coro}{Corollary}{Corollaries}
\crefname{dfn}{definition}{definitions}
\Crefname{dfn}{Definition}{Definitions}
\crefname{rem}{remark}{remarks}
\Crefname{rem}{Remark}{Remarks}
\crefname{alg}{algorithm}{algorithms}
\Crefname{alg}{Algorithm}{Algorithms}

\newcommand*{\myproofname}{Proof}

\makeatletter
\newcommand{\neutralize}[1]{\expandafter\let\csname c@#1\endcsname\count@}
\makeatother

\renewcommand{\d}{\mathrm{d}}

\newcommand{\R}{\mathbb{R}}

\newcommand{\defeq}{\coloneqq}

\newcommand{\emphd}[1]{{\fontseries{b}\selectfont\textsf{#1}}}
\newcommand{\W}{\mathcal{W}}

\newcommand{\one}{\mathbf{1}}
\newcommand{\inp}[1]{\langle #1 \rangle}

\DeclareMathOperator{\tr}{tr}
\numberwithin{equation}{section}

\begin{document}

\title{Real-analytic finitely forcible kernels}
\date{}
\author{
  Junchi Zhang
}

\maketitle
\vspace{-2em}

\begin{abstract}
Lov\'asz and Szegedy asked whether a nonconstant continuous, or even smooth,
finitely forcible kernel exists.  For every $0<\lambda\leq1/128$, we
construct an explicit nonconstant real-analytic kernel $W_\lambda$ taking values
in $(1/4,7/8)$.  A single finite family of simple graphs, independent of
$\lambda$, forces every $W_\lambda$ among all bounded symmetric real-valued kernels on $[0,1]^2$.
\end{abstract}

\section{Introduction}

Graph limit theory provides a framework for recovering global structure from local statistics and brings probabilistic and analytic tools to the study of large graphs.
The adjacency matrix of a finite graph may be
viewed as a black-and-white pixel image.  Along a convergent sequence of dense graphs whose number of vertices tends to infinity, these images have increasing resolution and after suitable relabellings of the vertices, their limiting profile may be represented by a function on $[0,1]^2$.  This limit object is called a \emph{graphon} \cite{LovaszSzegedy2006,BorgsChayesLovaszSosVesztergombi2008,Lovasz2012}.

Formally, a \emphd{kernel} is a bounded symmetric measurable function from $[0,1]^2$ to $\R$, and a \emphd{graphon} is a kernel taking values in $[0,1]$.
We use $\W$ to denote the space of kernels and $\W_0$ the space of graphons.
Throughout, kernels that agree almost everywhere are identified.
For a kernel $W$ and a finite simple graph $F$, its \emphd{homomorphism density} in $W$ is
\[
    t(F,W)
    =
    \int_{[0,1]^{V(F)}}
    \prod_{ij\in E(F)}W(x_i,x_j)
    \prod_{i\in V(F)}\,\d x_i.
\]
Two kernels are \emphd{weakly isomorphic} if their homomorphism densities agree
for every finite simple graph \cite[Section 7.3]{Lovasz2012}.

An element $W\in \mathcal{X}$, $\mathcal{X}=\W$ or $\W_0$, is called \emphd{finitely forcible} in $\mathcal{X}$ if there are finite simple graphs
$F_1,\ldots,F_m$ such that every $U\in \mathcal{X}$ satisfying
\[
    t(F_i,U)=t(F_i,W),
    \qquad i=1,\ldots,m,
\]
is weakly isomorphic to $W$.
Finitely forcible graphons arise naturally as
unique solutions of extremal problems involving finitely many subgraph
densities; see \cite{CooperKralMartins2018,GrzesikKral2019}.

Classical quasirandomness results show that constant graphons are finitely forcible in $\W_0$ \cite{Thomason1987,ChungGrahamWilson1989};
Lov\'asz and S\'os proved the same for every step graphon \cite{LovaszSos2008};
Grzesik, Kr\'{a}l', and Pikhurko extended this result to $\W$ \cite{GrzesikKralPikhurko2024}.
Later, Lov\'asz and Szegedy initiated the systematic study
of finite forcibility and constructed non-step examples of algebraic and iterated type, including examples with continuum-sized range \cite{LovaszSzegedy2011}.
They also proved that every finitely forcible graphon is either a step graphon or has infinite rank
\cite[Corollary~7.7]{LovaszSzegedy2011}.

Later constructions showed that finitely forcible graphons may have spaces of typical vertices that are not locally compact or are infinite-dimensional \cite{GlebovKralVolec2019,GlebovKlimosovaKral2019},
and may have weak-regularity partitions of nearly extremal complexity \cite{CooperKaiserKralNoel2018}.
More generally, every graphon occurs on a positive-measure part of a finitely forcible graphon \cite{CooperKralMartins2018},
and that part may have measure arbitrarily close to one \cite{KralLovaszNoelSosnovec2020}.
These results show that finitely forcible graphons can be quite complex, but do not give a nonconstant finitely forcible graphon that is continuous on the whole square.

Lov\'asz and Szegedy asked whether such a nonconstant
continuous, or even smooth, finitely forcible function on $[0,1]^2$ exists \cite[Question~1]{LovaszSzegedy2011}.
We answer this question in the affirmative, in the stronger real-analytic form.
A kernel is called \emphd{real analytic} if it extends to a real-analytic function on an open neighborhood of $[0,1]^2$.
Our main result is the following.

\begin{thm}\label{thm:main}
For $0<\lambda\leq1/128$, define
\[
    W_\lambda(x,y)=
    \frac12+\frac{\cos(\pi x)+\cos(\pi y)}8
    +\frac{\cos(\pi x)\cos(\pi y)}{16}+\sum_{n=2}^{\infty} \lambda^n\cos(n\pi x)\cos(n\pi y).
\]
Then $W_\lambda$ is a nonconstant real-analytic finitely forcible kernel in $\W$. Moreover, the kernels
$W_\lambda$ are pairwise non-weakly-isomorphic as $\lambda$ varies.
\end{thm}
The constraint $\lambda\leq 1/128$ is a convenient way to ensure some inequalities in the proof, and is not optimized. For every $\lambda$ in the stated range,
\[
    \frac14<W_\lambda(x,y)<\frac78,
    \qquad x,y\in[0,1],
\]
so these kernels are graphons bounded away from both $0$ and $1$.
The proof produces a finite family $\mathscr F$ of simple graphs, independent of $\lambda$, such that the densities
\[
    t(F,U)=t(F,W_\lambda),
    \qquad F\in\mathscr F,
\]
determine $W_\lambda$.
More precisely,
there is a measure-preserving map
$\varphi\colon[0,1]\to[0,1]$ such that
\[
    U(x,y)=W_\lambda(\varphi(x),\varphi(y))
\]
for almost every $(x,y)\in[0,1]^2$.

\paragraph{Organization.}
Section~2 introduces basic facts about kernels, integral operators, and labelled graphs; Section~3 studies properties of $W_\lambda$;
Section~4 constructs the finite forcing family and proves \Cref{thm:main}.

\section{Preliminaries}

\subsection{Kernels and operators}

We use $\operatorname{Leb}$ to denote Lebesgue measure on $[0,1]$.
All $L^2$ spaces in this paper are over $\R$; in particular, $L^2[0,1]$
is a real Hilbert space.
A measurable map $\varphi\colon[0,1]\to[0,1]$ is
\emphd{measure preserving} if for every measurable set $A\subseteq[0,1]$,
\[
 \operatorname{Leb}(\varphi^{-1}(A))=\operatorname{Leb}(A).
\]
For a kernel $W$, its \emphd{pullback} by $\varphi$ is the kernel
\[
 W^\varphi(x,y)=W(\varphi(x),\varphi(y)).
\]
For every finite simple graph $F$, applying measure preservation in each vertex coordinate gives
\[
 t(F,W^\varphi)=t(F,W).
\]
Hence, $W^\varphi$ is weakly isomorphic
to $W$.  More generally, by \cite[Corollary~10.35(a)]{Lovasz2012}, two kernels $U$ and $W$ are weakly isomorphic if and only if there exist measure-preserving maps
$\varphi,\psi\colon[0,1]\to[0,1]$ such that
\[
 U^\varphi(x,y)=W^\psi(x,y),
 \qquad\text{a.e. }(x,y)\in[0,1]^2.
\]

For a kernel $W$, we define
\[
    d_W(x)=\int_0^1W(x,y)\,\d y,
    \qquad
    C_W(x,y)=\int_0^1W(x,z)W(z,y)\,\d z,
    \qquad
    r_W(x)=\int_0^1C_W(x,y)\,\d y.
\]
The integral operator associated with $W$ is
\[
    (T_Wf)(x)=\int_0^1W(x,y)f(y)\,\d y,
    \qquad f\in L^2[0,1].
\]
Further, if an operator $T\colon L^2[0,1]\to L^2[0,1]$ is represented in this way by a kernel $W$, then $W$ is called an \emphd{integral kernel} representing $T$.

The following standard facts about operators may be found in
\cite[Chapters~1--3]{Simon2005}; see also
\cite[Section~7.5]{Lovasz2012}.
Given an orthonormal basis $(e_n)_{n=0}^\infty$ for $L^2[0,1]$, an operator $T\colon L^2[0,1]\to L^2[0,1]$ is called
\emphd{Hilbert--Schmidt} if
\[
    \|T\|_{\mathrm{HS}}^2
    \defeq
    \sum_{n=0}^\infty\|Te_n\|_2^2
    <\infty;
\]
this condition and the value of the sum are independent of the chosen basis.
For an integral operator $T_K$ associated with a kernel $K$, the Hilbert--Schmidt condition is equivalent to $K\in L^2([0,1]^2)$, and
$\|T_K\|_{\mathrm{HS}}=\|K\|_{L^2}$.
A Hilbert--Schmidt operator is compact.

An operator $T\colon L^2[0,1]\to L^2[0,1]$ is called \emphd{self-adjoint} if $\inp{Tf,g}=\inp{f,Tg}$ for all $f,g \in L^2[0,1]$.
A self-adjoint operator $T$ is called \emphd{positive} if
$\langle Tf,f\rangle\geq0$ for every $f\in L^2[0,1]$, and it is called \emphd{injective} if its
null space is $\{0\}$.

We call $\alpha$ an \emphd{eigenvalue} of $T$ if $Tf=\alpha f$ for some nonzero $f\in L^2[0,1]$.
If $T$ is compact and self-adjoint, its nonzero eigenvalues form a finite or countable sequence $(\alpha_i)_{i\in I}$, each repeated according to its
finite multiplicity. Thus the
operator $T$ is Hilbert--Schmidt if and only if
$\sum_{i\in I} \alpha_i^2<\infty$.

A compact self-adjoint operator $T$ is called \emphd{trace class} if
$\sum_{i\in I} |\alpha_i|<\infty$. 
Thus if $T$ is trace class, then
\[
 \tr(T)\defeq \sum_{n=0}^\infty \inp{Te_n,e_n} =\sum_{i\in I}\alpha_i,
\]
and this series is absolutely convergent. In particular, the product of two Hilbert--Schmidt self-adjoint operators is trace
class.

If $T$ is positive and bounded, then $T$ has a unique positive square root $T^{1/2}$; if $T$ is compact and self-adjoint with nonzero eigenvalues $(\alpha_i)_{i\in I}$, we write $|T|=(T^2)^{1/2}$, whose nonzero eigenvalues are $(|\alpha_i|)_{i\in I}$.

Every kernel $W$ belongs to $L^2([0,1]^2)$, so $T_W$ is Hilbert--Schmidt
and hence compact; symmetry of $W$ makes $T_W$ self-adjoint.
Let $P_3$ denote the path on three vertices and $C_3$ the triangle.  We record the following basic facts.

\begin{lem}\label[lem]{lem:operator-identities}
For every kernel $W$, we have
\begin{align*}
    (T_W^2f)(x)&=\int_0^1C_W(x,y)f(y)\,\d y,\\
    t(K_2,W)&=\int_0^1d_W(x)\,\d x,  \qquad  t(C_3,W)=\tr(T_W^3),\\
    t(P_3,W)&=\int_0^1d_W(x)^2\,\d x
      =\int_0^1r_W(x)\,\d x
      =\int_{[0,1]^2}C_W(x,y)\,\d x\,\d y.
\end{align*}
Here $T_W^3$ is trace class, even though $T_W$ itself need not be.
\end{lem}

\begin{proof}
Fubini's theorem gives
\[
    (T_W^2f)(x)
    =
    \int_0^1
    \left(\int_0^1W(x,z)W(z,y)\,\d z\right)f(y)\,\d y,
\]
which proves the first assertion.
The identity for $K_2$ is its definition.
The identity for $C_3$ is the cycle trace formula
\cite[Section~7.5, (7.22)]{Lovasz2012}.
For the path $P_3$,
\begin{align*}
    \int_{[0,1]^2}C_W(x,y)\,\d x\,\d y
    &=
    \int_0^1
      \left(\int_0^1W(x,z)\,\d x\right)
      \left(\int_0^1W(z,y)\,\d y\right)\,\d z\\
    &=\int_0^1d_W(z)^2\,\d z
    =t(P_3,W).
\end{align*}
The equality with $\int_0^1r_W\,\d x$ follows from the definition of $r_W$.
\end{proof}

We can reduce the problem of forcing a kernel to the problem of forcing its codegree function.

\begin{lem}\label[lem]{lem:positive-square-root}
Let $U$ and $W$ be kernels, and let
$\varphi\colon[0,1]\to[0,1]$ be measure preserving.  Suppose that $T_W$ is
positive,
\[
    C_U(x,y)=C_W(\varphi(x),\varphi(y)),
    \qquad\text{a.e. }(x,y)\in[0,1]^2,
\]
and $t(C_3,U)=t(C_3,W)$.  Then
\[
    U(x,y)=W(\varphi(x),\varphi(y)),
    \qquad\text{a.e. }(x,y)\in[0,1]^2.
\]
\end{lem}

\begin{proof}
Put $V(x,y)=W(\varphi(x),\varphi(y))$.  Then
\[
    C_V(x,y)=C_W(\varphi(x),\varphi(y))=C_U(x,y),
    \quad\text{a.e. }(x,y)\in[0,1]^2,
\]
and $t(C_3,V)=t(C_3,W)$.
For $f\in L^2[0,1]$, let $\nu$ be the finite signed measure defined by
\[
    \nu(A)=\int_{\varphi^{-1}(A)}f(x)\,\d x.
\]
Measure preservation gives $\nu\ll\operatorname{Leb}$.  Moreover, for every
bounded measurable $h$, measure preservation and Cauchy--Schwarz give
\[
 \left|\int h\,\d\nu\right|=\left|\int (h\circ\varphi)f\,\d x\right|
 \leq\lVert h\rVert_2\lVert f\rVert_2.
\]
Hence, by $L^2$-duality, the density $g=\d\nu/\d\operatorname{Leb}$ belongs to $L^2[0,1]$.
Using the definition of $\nu$ gives
\[
\begin{aligned}
    \inp{T_Vf,f}
    &=\int_{[0,1]^2}W(\varphi(x),\varphi(y))f(x)f(y)\,\d x\,\d y\\
    &=\int_{[0,1]^2}W(s,t)\,\d\nu(s)\,\d\nu(t)=\inp{T_Wg,g}\geq0.
\end{aligned}
\]
Thus $T_V$ is positive.  The
hypotheses and \Cref{lem:operator-identities} give $T_U^2=T_V^2$ and
$t(C_3,U)=t(C_3,V)$.  Uniqueness of the positive square root therefore gives
\[
    T_V=(T_V^2)^{1/2}=(T_U^2)^{1/2}=|T_U|.
\]

Let $(\theta_i)_{i\in I}$ be the nonzero eigenvalues of $T_U$, repeated with
multiplicity.  The operators $T_U^3$ and $|T_U|^3$ are trace class, and hence
\[
  0=t(C_3,V)-t(C_3,U)
  =\tr(|T_U|^3)-\tr(T_U^3)
  =2\sum_{\theta_i<0,\;i\in I}|\theta_i|^3.
\]
Thus $T_U$ is positive and $T_U=|T_U|=T_V$.  Consequently,
\[
    \|U-V\|_{L^2}=\|T_U-T_V\|_{\mathrm{HS}}=0,
\]
which proves the asserted identity.
\end{proof}

\subsection{The labelled graph algebra}\label{sec:labelled-graphs}

We give some basic definitions and properties of the labelled quantum graph algebra used later; see also~\cite[Section~2.2]{LovaszSzegedy2011} and~\cite[Chapter~6]{Lovasz2012}.
A \emphd{$k$-labelled graph} is a finite loopless 
multigraph with $k$ distinguished vertices carrying distinct labels $1,\ldots,k$.
We identify these vertices with $[k]=\{1,\ldots,k\}$.
The remaining vertices are unlabelled.
For a kernel $U$ and $\mathbf{x}=(x_1,\ldots,x_k)\in [0,1]^k$, the $k$-\emphd{labelled homomorphism density} of a $k$-labelled graph $F$ in $U$ at $\mathbf{x}$ is
defined as
\[
    t_k(F,U)(\mathbf{x})
    =
    \int_{[0,1]^{V(F)\setminus[k]}}
       \prod_{uv\in E(F)}U(z_u,z_v)
       \prod_{v\in V(F)\setminus[k]}\,\d z_v,
\]
where $z_i=x_i$ at a labelled vertex.
Two $k$-labelled homomorphism densities are identified if they agree almost everywhere on $[0,1]^k$.
A \emphd{labelled quantum graph} is a finite real linear combination of labelled graphs, and $t_k$ is extended linearly.
For $k=0$ this is the ordinary homomorphism density $t=t_0$; the density of an unlabelled quantum graph is likewise defined by linear extension.
For the graph $\varnothing$ with no vertices, we use the convention
$t(\varnothing,U)=1$.

Given two $k$-labelled graphs $F$ and $G$, their product $FG$ is the $k$-labelled graph obtained from the disjoint union of $F$ and $G$ by identifying vertices carrying the same label.  The product is extended linearly to labelled quantum graphs. For fixed $\mathbf{x}\in [0,1]^k$, the two sets of integration variables are disjoint, so
\[
    t_k(FG,U)(\mathbf{x})
    =
    t_k(F,U)(\mathbf{x})t_k(G,U)(\mathbf{x}),\qquad \text{a.e. }\mathbf{x}\in [0,1]^k.
\]
Let $\llbracket h\rrbracket$ denote the unlabelled quantum graph obtained by forgetting all labels.
The following identity is a classical result, and for completeness, we give a proof.

\begin{lem}\label[lem]{lem:labelled-square}
For every $k$-labelled quantum graph $h$ and every kernel $U$,
\[
    t(\llbracket h^2\rrbracket,U)
    =
    \int_{[0,1]^k}t_k(h,U)(\mathbf{x})^2\,\d\mathbf{x}.
\]
\end{lem}

\begin{proof}
Unlabelling, multiplicativity, and Fubini's theorem give
\[
    t(\llbracket h^2\rrbracket,U)
    =
    \int_{[0,1]^k}t_k(h^2,U)(\mathbf{x})\,\d\mathbf{x}
    =
    \int_{[0,1]^k}t_k(h,U)(\mathbf{x})^2\,\d\mathbf{x}.\hfill \qedhere
\]
\end{proof}

We finish by defining several special labelled graphs.  Let $\mathcal O_1$
be the edgeless graph consisting only of one labelled vertex.  Let $P_2^1$
be the one-labelled path on vertices $\{1,u\}$, with label $1$ at one
endpoint and $u$ unlabelled, and let $P_3^1$ be the one-labelled two-edge
path with its label at one endpoint.  Then for a.e. $x\in [0,1]$,
\begin{equation}\label{eq:rooted-labelled-evaluations}
    t_1(\mathcal O_1,U)(x)=1,
    \qquad
    t_1(P_2^1,U)(x)=d_U(x),
    \qquad
    t_1(P_3^1,U)(x)=r_U(x).
\end{equation}

\begin{center}
\begin{tikzpicture}[
    vertex/.style={circle,fill,minimum size=2mm,inner sep=0pt},
    every label/.style={font=\small},
    label distance=1pt,
    every path/.style={line width=0.5pt}
]
  \begin{scope}[xshift=0cm]
    \node[vertex,label=left:$1$] (o1) at (0,0) {};
    \node at (0,-0.75) {$\mathcal O_1$};
  \end{scope}
  \begin{scope}[xshift=3.2cm]
    \node[vertex,label=left:$1$] (p21) at (-0.55,0) {};
    \node[vertex] (p22) at (0.55,0) {};
    \draw (p21)--(p22);
    \node at (0,-0.75) {$P_2^1$};
  \end{scope}
  \begin{scope}[xshift=7cm]
    \node[vertex,label=left:$1$] (p31) at (-1,0) {};
    \node[vertex] (p32) at (0,0) {};
    \node[vertex] (p33) at (1,0) {};
    \draw (p31)--(p32)--(p33);
    \node at (0,-0.75) {$P_3^1$};
  \end{scope}
\end{tikzpicture}
\end{center}

Let $\mathcal O$ be the edgeless graph consisting only of the two labelled
vertices.
For $i\in\{1,2\}$, let $\mathcal A^i$ be the two-labelled graph consisting
of an edge incident with label $i$ and an isolated vertex carrying the other
label.
Let $P_3^{12}$ be the path $1$--$z$--$2$, whose endpoints carry labels
$1,2$ and whose middle vertex $z$ is unlabelled.
Then for a.e. $(x,y)\in [0,1]^2$,
\begin{equation}\label{eq:basic-labelled-evaluations}
\begin{aligned}
    t_2(\mathcal O,U)(x,y)&=1,
    &\qquad t_2(\mathcal A^1,U)(x,y)&=d_U(x),\\
    t_2(\mathcal A^2,U)(x,y)&=d_U(y),
    &\qquad t_2(P_3^{12},U)(x,y)&=C_U(x,y).
\end{aligned}
\end{equation}

\begin{center}
\begin{tikzpicture}[
    vertex/.style={circle,fill,minimum size=2mm,inner sep=0pt},
    every label/.style={font=\small},
    label distance=1pt,
    every path/.style={line width=0.5pt}
]
  \begin{scope}[xshift=0cm]
    \node[vertex,label=left:$1$] at (-0.65,0) {};
    \node[vertex,label=right:$2$] at (0.65,0) {};
    \node at (0,-0.7) {$\mathcal O$};
  \end{scope}
  \begin{scope}[xshift=3.4cm]
    \node[vertex,label=left:$1$] (a11) at (-0.65,0) {};
    \node[vertex,label=right:$2$] at (0.65,0) {};
    \node[vertex] (a1u) at (0,0.75) {};
    \draw (a11)--(a1u);
    \node at (0,-0.7) {$\mathcal A^1$};
  \end{scope}
  \begin{scope}[xshift=6.8cm]
    \node[vertex,label=left:$1$] at (-0.65,0) {};
    \node[vertex,label=right:$2$] (a22) at (0.65,0) {};
    \node[vertex] (a2u) at (0,0.75) {};
    \draw (a22)--(a2u);
    \node at (0,-0.7) {$\mathcal A^2$};
  \end{scope}
  \begin{scope}[xshift=10.5cm]
    \node[vertex,label=left:$1$] (q1) at (-0.65,0) {};
    \node[vertex] (qu) at (0,0.75) {};
    \node[vertex,label=right:$2$] (q2) at (0.65,0) {};
    \draw (q1)--(qu)--(q2);
    \node at (0,-0.7) {$P_3^{12}$};
  \end{scope}
\end{tikzpicture}
\end{center}

\section{Basic properties of \texorpdfstring{$W_\lambda$}{W-lambda}}

Fix $0<\lambda\leq1/128$.  We use the cosine orthonormal basis
$(e_n)_{n\geq0}$ of $L^2[0,1]$:
\[
 e_0(x)=1,\qquad e_n(x)=\sqrt2\cos(n\pi x),\qquad n\geq1.
\]
Orthogonality of the cosine basis gives
\begin{equation}\label{eq:operator-decomposition}
 L^2[0,1]
 =
 \operatorname{span}\{e_0,e_1\}
 \oplus\bigoplus_{n=2}^{\infty}\operatorname{span}\{e_n\},
 \qquad
 T_{W_\lambda}
 =
 B_0\oplus\bigoplus_{n=2}^{\infty}
 \frac{\lambda^n}{2}
 \operatorname{Id}_{\operatorname{span}\{e_n\}},
\end{equation}
where $B_0$ acts on $\operatorname{span}\{e_0,e_1\}$ and, in the ordered
basis $(e_0,e_1)$,
\[
    B_0
    =
    \begin{pmatrix}
      1/2&1/(8\sqrt2)\\
      1/(8\sqrt2)&1/32
    \end{pmatrix}.
\]

The following two propositions record some basic homomorphism densities of $W_\lambda$.

\begin{prop}\label[prop]{prop:target-properties}
For $0<\lambda\leq1/128$, the function $W_\lambda$ takes values in $(1/4,7/8)$
and extends to a real-analytic function on $\R^2$.  It is nonconstant and
\begin{align}
 d_{W_\lambda}(x)
 &=\frac12+\frac18\cos(\pi x),
 \label{eq:target-degree}\\
 r_{W_\lambda}(x)
 &=\frac{33}{128}+\frac{17}{256}\cos(\pi x)
 =\frac{17}{32}d_{W_\lambda}(x)-\frac1{128}.
 \label{eq:target-rooted-relation}
\end{align}
Further, the operator $T_{W_\lambda}$ is positive and injective.
\end{prop}

\begin{proof}
For $x,y\in[0,1]$, write $a=\cos(\pi x)$ and $b=\cos(\pi y)$.  Then
\[
 \frac12+\frac{a+b}{8}+\frac{ab}{16}
 =\frac14+\frac{(2+a)(2+b)}{16}
 \in\left[\frac5{16},\frac{13}{16}\right],
 \qquad
 \left|\sum_{n=2}^{\infty}
 \lambda^n\cos(n\pi x)\cos(n\pi y)\right|
 \leq\frac{\lambda^2}{1-\lambda}<\frac1{16}.
\]
Thus $1/4<W_\lambda(x,y)<7/8$.

Set
\[
 G_\lambda(t)
 \defeq
 \sum_{n=2}^{\infty}\lambda^n\cos(nt)
 =\Re\left(\sum_{n=2}^\infty \lambda^n e^{int} \right)
 =
 \frac{\lambda^2\cos(2t)-\lambda^3\cos t}
      {1-2\lambda\cos t+\lambda^2}.
\]
Since
\[
 1-2\lambda\cos t+\lambda^2\geq(1-\lambda)^2>0,
\]
the function $G_\lambda$ is real analytic on $\mathbb R$.  The product-to-sum
identity gives
\[
 \sum_{n=2}^{\infty}
   \lambda^n\cos(n\pi x)\cos(n\pi y)
 = \frac12\bigl(
   G_\lambda(\pi(x-y))+G_\lambda(\pi(x+y))
 \bigr).
\]
The right-hand side extends $W_\lambda$ to a real-analytic function on
$\mathbb R^2$.

Next, the definitions of $d_{W_\lambda}$ and
$r_{W_\lambda}$ give
\[
 d_{W_\lambda}=T_{W_\lambda}e_0,
 \qquad
 r_{W_\lambda}=T_{W_\lambda}^2e_0.
\]
Therefore \eqref{eq:operator-decomposition} yields
\[
 d_{W_\lambda}
 =\frac12e_0+\frac1{8\sqrt2}e_1
 =\frac12+\frac18\cos(\pi x),
\]
\begin{equation}\label{eq:low-block-square}
 B_0^2=
 \begin{pmatrix}
  33/128&17/(256\sqrt2)\\
  17/(256\sqrt2)&9/1024
 \end{pmatrix},
\end{equation}
and
\[
 r_{W_\lambda}
 =\frac{33}{128}e_0+\frac{17}{256\sqrt2}e_1
 =\frac{33}{128}+\frac{17}{256}\cos(\pi x).
\]
This also proves the last equality in \eqref{eq:target-rooted-relation}.
Since $d_{W_\lambda}$ is nonconstant, so is $W_\lambda$.

Finally, any $f\in L^2[0,1]$ has the decomposition
\[
 f=a_0e_0+a_1e_1+\sum_{n=2}^{\infty}a_ne_n.
\]
The decomposition \eqref{eq:operator-decomposition} gives
\[
\langle T_{W_\lambda}f,f\rangle=
\frac12\left(a_0+\frac{a_1}{4\sqrt2}\right)^2
+\frac{a_1^2}{64}
+\sum_{n=2}^{\infty}\frac{\lambda^n}{2}a_n^2.
\]
This is strictly positive for $f\neq0$, proving positivity and injectivity.
\end{proof}

\begin{prop}
For every $0<\lambda\leq1/128$,
\begin{align}
    t(K_2,W_\lambda)&=\frac12,\notag\\
    t(P_3,W_\lambda)&=\frac{33}{128},\label{eq:target-path-density}\\
    t(C_3,W_\lambda)
    &=
    \frac{4505}{32768}
    +\frac{\lambda^6}{8(1-\lambda^3)}.
    \label{eq:target-triangle-density}
\end{align}
\end{prop}

\begin{proof}
    By \Cref{lem:operator-identities} and
\eqref{eq:target-degree},
\[
\begin{aligned}
t(K_2,W_\lambda)
&=\int_0^1 d_{W_\lambda}(x)\,\d x
 =\int_0^1
   \left(\frac12+\frac18\cos(\pi x)\right)\,\d x
 =\frac12,
\\
t(P_3,W_\lambda)
&=\int_0^1 d_{W_\lambda}(x)^2\,\d x =\int_0^1
  \left(\frac12+\frac18\cos(\pi x)\right)^2\,\d x =\frac{33}{128}. 
\end{aligned}
\]

For the triangle, combining
\Cref{lem:operator-identities} with \eqref{eq:operator-decomposition} gives
\[
 \tr(B_0^3)
 =
 \left(\frac12\right)^3+\left(\frac1{32}\right)^3
 +3\left(\frac1{8\sqrt2}\right)^2
   \left(\frac12+\frac1{32}\right)
 =\frac{4505}{32768},
\]
and therefore
\[
 t(C_3,W_\lambda)=\tr(T_{W_\lambda}^3)
 =\frac{4505}{32768}+\frac18\sum_{n=2}^{\infty}\lambda^{3n}
 = \frac{4505}{32768}+\frac{\lambda^6}{8(1-\lambda^3)}.\hfill \qedhere
\]
\end{proof}

Our next goal is to express $C_{W_\lambda}$ as a rational function of $d_{W_\lambda}(x)$.
Let $P_n$ be the $n$-th \emphd{Chebyshev polynomial},
characterized by $P_n(\cos\theta)=\cos(n\theta)$.

For $z\in (0,1)$, set
\[
a_z=\frac{1+z}{2\sqrt z}.
\]
For $(s,t)\in \R^2$, set
\begin{align}
 L(s,t)
 &=\frac{33}{128}+\frac{17(s+t)}{256}+\frac{9st}{512},\notag\\
 V_z(s,t)
 &=(1-z)^2\bigl((1+z)^2-4zst\bigr)+4z^2(s-t)^2,
 \notag\\
 N_z(s,t)
 &=\frac12\Bigl(
 1-z^2+2z^2(s^2+t^2)-z(3+z^2)st
 -(1+zst)V_z(s,t)
 \Bigr),
 \notag\\
 I_z(s,t)
 &=V_z(s,t)L(s,t)+N_z(s,t).
 \label{eq:certificate-polynomial}
\end{align}

A series \emphd{converges locally uniformly} on a set $S$ if it converges uniformly on every compact subset of $S$. Instead of working with $d_{W_\lambda}$ directly, we express $C_{W_\lambda}$ as a rational function of
\[
\cos(\pi x)=8d_{W_\lambda}(x)-4.
\]

\begin{lem}\label[lem]{lem:chebyshev-kernel}
Fix $0<z<1$.  On $(-a_z,a_z)^2$, the series
\[
 \frac12\sum_{n=2}^{\infty}z^nP_n(s)P_n(t)
\]
converges locally uniformly and absolutely.  Define
\begin{equation}\label{eq:codegree-series}
 R_z(s,t)
 =L(s,t)+\frac12\sum_{n=2}^{\infty}z^nP_n(s)P_n(t).
\end{equation}
Then
\[
 V_z(s,t)>0,
 \qquad
 \frac12\sum_{n=2}^{\infty}z^nP_n(s)P_n(t)
 =\frac{N_z(s,t)}{V_z(s,t)}.
\]
Consequently,
\begin{equation}\label{eq:codegree-rationalization}
 R_z(s,t)
 =L(s,t)+\frac{N_z(s,t)}{V_z(s,t)}
 =\frac{I_z(s,t)}{V_z(s,t)}.
\end{equation}
If $0<z\leq2^{-14}$, then
\begin{equation}\label{eq:pole-signs}
\begin{aligned}
 V_z(s,s)&<0<N_z(s,s),&& |s|>a_z,\\
 V_z(s,s)&=0<N_z(s,s),&& |s|=a_z.
\end{aligned}
\end{equation}
Further, for $0<\lambda\leq1/128$,
\begin{equation}\label{eq:target-codegree}
 C_{W_\lambda}(x,y)
 =R_{\lambda^2}(\cos(\pi x),\cos(\pi y)),\qquad x,y\in[0,1].
\end{equation}
\end{lem}

\begin{proof}
Fix $0\leq b<a_z$ and put
\[
 c_b=
 \begin{cases}
  1,&b\leq1,\\
  b+\sqrt{b^2-1},&b>1.
 \end{cases}
\]
For $|u|>1$, the standard representation
\cite[Section~1.4.2, (1.47)]{MasonHandscomb2003} can be written as
\[
 P_n(u)=\frac{(u+\sqrt{u^2-1})^n+(u-\sqrt{u^2-1})^n}{2},
\]
If $1<|u|\leq b$, both quantities
$|u\pm\sqrt{u^2-1}|$ are at most $c_b$; while for $|u|\leq1$, the defining
identity $P_n(\cos\theta)=\cos(n\theta)$ gives $|P_n(u)|\leq1$.
Hence
\[
\sup_{|u|\leq b}|P_n(u)|\leq c_b^n.
\]
For $b\in (1,a_z)$, $zc_b^2<zc_{a_z}^2<1$.
For $b\in [0,1]$, $zc_b^2\leq z<1$.
Thus the series
\[
\dfrac{1}{2}\sum_{n=2}^\infty z^n P_n(s)P_n(t)
\]
converges uniformly and absolutely on $[-b,b]^2$.
Since every compact subset of $(-a_z,a_z)^2$ is contained in $[-b,b]^2$ for some
$b<a_z$, the series converges locally uniformly and absolutely on $(-a_z,a_z)^2$.

For $|s|,|t|\leq1$ and $|r|<1$, the bilinear generating formula
\cite[Proposition~1, (2.2)]{Szablowski2019}, after rewriting its
denominator, gives
\[
 S_r(s,t)\defeq\sum_{n=0}^{\infty}r^nP_n(s)P_n(t)
 =\frac{Q_r(s,t)}{V_r(s,t)},
 \qquad
 Q_r(s,t)=1-r^2+2r^2(s^2+t^2)-r(3+r^2)st.
\]
For each $m\geq0$, the coefficient of $r^m$ in
$V_rS_r-Q_r$ is a polynomial in $s,t$ that vanishes on $[-1,1]^2$,
and hence vanishes identically.  Thus
\[
 V_rS_r=Q_r
\]
as an identity in $\mathbb{R}[s,t][[r]]$.

Now fix $z\in(0,1)$.  For $|s|,|t|<a_z$, the convergence established
above allows us to evaluate this formal identity at $r=z$, giving
\[
 V_z(s,t)S_z(s,t)=Q_z(s,t).
\]
Moreover,
\[
 V_z(s,t)
 =(1-z)^2\bigl((1+z)^2-4zst\bigr)+4z^2(s-t)^2>0,
\]
because $st<a_z^2=(1+z)^2/(4z)$.  Therefore
\[
 S_z(s,t)=\frac{Q_z(s,t)}{V_z(s,t)}
\]
for every $z\in(0,1)$ and $|s|,|t|<a_z$.

Since $P_0=1$ and $P_1(u)=u$, it follows that
\[
\begin{aligned}
 \frac12\sum_{n=2}^{\infty}z^nP_n(s)P_n(t)
 &=\frac12\left(S_z(s,t)-1-zst\right)\\
 &=\frac{Q_z(s,t)-(1+zst)V_z(s,t)}
         {2V_z(s,t)}
 =\frac{N_z(s,t)}{V_z(s,t)}.
\end{aligned}
\]
Thus,
\[
 R_z(s,t)
 =L(s,t)+\frac{N_z(s,t)}{V_z(s,t)}
 =\frac{V_z(s,t)L(s,t)+N_z(s,t)}{V_z(s,t)}
 =\frac{I_z(s,t)}{V_z(s,t)}.
\]

Suppose that $0<z\leq2^{-14}$.  On the diagonal,
\[
 V_z(s,s)=(1-z)^2\bigl((1+z)^2-4zs^2\bigr),
 \qquad
 N_z(s,s)=\frac{z^2(1-z)}2P_z(s^2),
\]
where
\[
 P_z(u)=4(1-z)u^2+(z^2+z-4)u+1+z.
\]
For $u\geq a_z^2$,
\[
 P_z(a_z^2)=\frac{(1-z)^2(1+z)}{4z^2}>0,
 \qquad
 P_z'(u)\geq P_z'(a_z^2)
 =\frac2z-2-z-z^2>0.
\]
Thus $P_z(u)>0$ for $u\geq a_z^2$, and
\eqref{eq:pole-signs} follows.

Finally, squaring \eqref{eq:operator-decomposition} gives
\[
 T_{W_\lambda}^2
 =
 B_0^2\oplus\bigoplus_{n=2}^{\infty}
 \frac{\lambda^{2n}}4
 \operatorname{Id}_{\operatorname{span}\{e_n\}}.
\]
By \eqref{eq:low-block-square}, the first summand has integral kernel
$L(\cos(\pi x),\cos(\pi y))$.  Since
$e_n(x)=\sqrt2P_n(\cos(\pi x))$, the right-hand side is represented by
\[
 L(\cos(\pi x),\cos(\pi y))
 +\frac12\sum_{n=2}^{\infty}\lambda^{2n}
 P_n(\cos(\pi x))P_n(\cos(\pi y))
 =
 R_{\lambda^2}(\cos(\pi x),\cos(\pi y)).
\]
By \Cref{lem:operator-identities}, this agrees almost everywhere with
$C_{W_\lambda}$.  Both sides are continuous (for $C_{W_\lambda}$ this follows
from dominated convergence), so they agree everywhere, proving \eqref{eq:target-codegree}.
\end{proof}

The following lemma helps to force $D_U$ uniformly bounded for every kernel $U$ we will consider and make the representing series uniformly converge.

\begin{lem}\label[lem]{lem:pole-exclusion}
Let $0<z\leq2^{-14}$, and let $D\colon[0,1]\to\R$ be bounded and
measurable.  Suppose that $K$ is a kernel, $T_K$ is positive, and
\begin{equation}\label{eq:residual-equation}
 V_z(D(x),D(y))K(x,y)=N_z(D(x),D(y))
 \qquad\text{for a.e. }(x,y)\in[0,1]^2.
\end{equation}
Then there is a number $0\leq b<a_z$ such that $|D|\leq b$ almost everywhere.
Moreover,
\[
 K(x,y)=\frac12\sum_{n=2}^{\infty}
 z^nP_n(D(x))P_n(D(y))
 \qquad\text{for a.e. }(x,y)\in[0,1]^2,
\]
and the series converges uniformly and absolutely on $[-b,b]^2$.
\end{lem}

\begin{proof}
Let $S$ be the essential range of $D$; thus $s\in S$ precisely when every open interval
$J$ containing $s$ satisfies $\operatorname{Leb}(D^{-1}(J))>0$.
Since $D$ is bounded, $S$ is compact.  Put $M=\lVert K\rVert_\infty$.
After changing $K$ on a null set, we may assume that
$|K(x,y)|\leq M$ for every $(x,y)\in[0,1]^2$.

Suppose first that $s\in S$ and $|s|>a_z$.  By \eqref{eq:pole-signs},
$N_z(s,s)/V_z(s,s)<0$.
Hence there are an interval $J$ containing $s$
and a number $\eta>0$ such that $V_z$ does not vanish on $J^2$ and
\[
 \frac{N_z(u,v)}{V_z(u,v)}\leq-\eta,
 \qquad u,v\in J.
\]
The set $A=\{x:D(x)\in J\}$ has positive measure.  Equation
\eqref{eq:residual-equation} gives
\[
 \langle T_K\one_A,\one_A\rangle
 =\int_{A\times A}K(x,y)\,\d x\,\d y
 \leq-\eta\operatorname{Leb}(A)^2<0,
\]
contrary to the positivity of $T_K$.

Suppose next that $s\in S$ and $|s|=a_z$.
Then $V_z(s,s)=0<N_z(s,s)$.
By continuity, there is an interval $J$ containing
$s$ such that
\[
 N_z(u,v)>M|V_z(u,v)|,
 \qquad u,v\in J.
\]
The set $A=\{x:D(x)\in J\}$ has positive measure.
As a consequence, for almost every
$(x,y)\in A^2$,
\[
 |V_z(D(x),D(y))K(x,y)|
 \leq M|V_z(D(x),D(y))|<N_z(D(x),D(y)),
\]
contradicting \eqref{eq:residual-equation} on a set of positive measure.
Thus $S\subset(-a_z,a_z)$.
By compactness, $|D|\leq b$ almost everywhere for
some $0\leq b<a_z$.

Now \Cref{lem:chebyshev-kernel} gives $V_z>0$ on $[-b,b]^2$ and
\[
 \frac{N_z(s,t)}{V_z(s,t)}
 =\frac12\sum_{n=2}^{\infty}z^nP_n(s)P_n(t),
 \qquad |s|,|t|\leq b,
\]
with uniform absolute convergence.
Dividing \eqref{eq:residual-equation} by $V_z(D(x),D(y))$ proves the asserted expansion.
\end{proof}

\section{Proof of the main theorem}

\begin{proof}
Fix $\lambda\in (0,1/128]$.
Recall the labelled graphs defined in \Cref{sec:labelled-graphs}
and, for a kernel $U$, put
\[
 D_U=8d_U-4.
\]
Define the one-labelled quantum graph
\[
 g=P_3^1-\frac{17}{32}P_2^1+\frac1{128}\mathcal O_1.
\]
Equations~\eqref{eq:rooted-labelled-evaluations} and
\eqref{eq:target-rooted-relation} give
\[
    t_1(g,W_\lambda)=0.
\]

Define a two-labelled quantum graph
\[
\begin{aligned}
 f_\lambda={}&
 V_{\lambda^2}(8\mathcal A^1-4\mathcal O,
    8\mathcal A^2-4\mathcal O)P_3^{12}-I_{\lambda^2}(8\mathcal A^1-4\mathcal O,
                    8\mathcal A^2-4\mathcal O).
\end{aligned}
\]
By \eqref{eq:basic-labelled-evaluations} and multiplicativity, for a.e. $(x,y)\in[0,1]^2$,
\begin{equation}\label{eq:certificate-evaluation}
 t_2(f_\lambda,U)(x,y)
 =
 V_{\lambda^2}(D_U(x),D_U(y))C_U(x,y)
 -I_{\lambda^2}(D_U(x),D_U(y)).
\end{equation}
Note that \eqref{eq:target-degree} gives
\[
    D_{W_\lambda}(x)=\cos(\pi x).
\]
Thus \eqref{eq:codegree-rationalization} and \eqref{eq:target-codegree} imply
\[
t_2(f_\lambda,W_\lambda)=0.
\]

Put
\[
 \mathcal Q_\lambda
 =\llbracket f_\lambda^2\rrbracket+\llbracket g^2\rrbracket.
\]
Since $V_{\lambda^2}(s,t)$ and $I_{\lambda^2}(s,t)$ are polynomials in $\lambda,s,t$ with fixed finite monomial supports, $f_\lambda$ is a linear combination of finitely many fixed two-labelled graphs, with coefficients depending polynomially on $\lambda$.
Since $g$ is fixed, the quantum graphs $\mathcal Q_\lambda$ lie in a single finite-dimensional vector space of unlabelled graphs.
Expand $\mathcal Q_\lambda$ in unlabelled graphs,
delete all isolated vertices and then collect isomorphic graphs.
Let $c_H(\lambda)$ be the resulting coefficient of $H$, and let $\mathscr B$ consist of the nonempty isomorphism classes for which $c_H$ is not identically zero.
Deleting isolated vertices does not change homomorphism densities.
Consequently, for every kernel $U$,
\begin{equation}\label{eq:forcing-expansion}
 t(\mathcal Q_\lambda,U)
 =c_{\varnothing}(\lambda)t(\varnothing,U)
  +\sum_{H\in\mathscr B}c_H(\lambda)t(H,U),
\end{equation}
where $t(\varnothing,U)=1$.

Every graph occurring in $\mathcal Q_\lambda$ is simple.  Indeed, products
identify only equally labelled vertices; every edge has an unlabelled endpoint,
and unlabelled vertices from distinct factors remain distinct, so no loop or
parallel edge is created.  Hence every graph in $\mathscr B$ is simple.  Finally, set
\[
 \mathscr F=\mathscr B\cup\{K_2,\;P_3,\;C_3\}.
\]
This is a finite family of simple graphs independent of $\lambda$.

Let $U$ be a kernel satisfying
\[
    t(H,U)=t(H,W_\lambda)
    \qquad\text{for every }H\in\mathscr F.
\]
We prove that $U$ is a pullback of $W_\lambda$.

By \Cref{lem:labelled-square}, the two identities already proved for
$W_\lambda$ give $t(\mathcal Q_\lambda,W_\lambda)=0$.  On the other hand,
\eqref{eq:forcing-expansion} and the density constraints give
\begin{align*}
 t(\mathcal Q_\lambda,U)
 &=c_{\varnothing}(\lambda)
   +\sum_{H\in\mathscr B}c_H(\lambda)t(H,U)\\
 &=c_{\varnothing}(\lambda)
   +\sum_{H\in\mathscr B}c_H(\lambda)t(H,W_\lambda)\\
 &=t(\mathcal Q_\lambda,W_\lambda)=0.
\end{align*}
Applying \Cref{lem:labelled-square} once more, together with
\eqref{eq:certificate-evaluation} and \eqref{eq:rooted-labelled-evaluations}, yields
\begin{align*}
0={}&\int_{[0,1]^2}
 \left(
 V_{\lambda^2}(D_U(x),D_U(y))C_U(x,y)
 -I_{\lambda^2}(D_U(x),D_U(y))
 \right)^2\,\d x\,\d y\\
&+\int_0^1
 \left(r_U(x)-\frac{17}{32}d_U(x)+\frac1{128}\right)^2\,\d x.
\end{align*}
Both summands are nonnegative.  Hence
\begin{align}
 V_{\lambda^2}(D_U(x),D_U(y))C_U(x,y)
 &=I_{\lambda^2}(D_U(x),D_U(y)),
 \qquad\text{for a.e. }(x,y)\in[0,1]^2,
 \label{eq:forced-codegree-identity}
\end{align}
and
\begin{align}
 r_U(x)&=\frac{17}{32}d_U(x)-\frac1{128},
 \qquad\text{for a.e. }x\in[0,1].
 \label{eq:forced-rooted-relation}
\end{align}

Expanding $D_U=8d_U-4$ and using the constraints for $K_2$ and $P_3$
gives
\[
 \int_0^1D_U(x)\,\d x=0,
 \qquad
 \int_0^1D_U(x)^2\,\d x=\frac12.
\]
Thus
\[
 e_0=\one,
 \qquad
 e_1=\sqrt2D_U
\]
are orthonormal.  Since $d_U=T_Ue_0$ and $r_U=T_Ud_U$,
\eqref{eq:forced-rooted-relation} gives
\[
 T_Ue_0=\frac12e_0+\frac1{8\sqrt2}e_1,
 \qquad
 T_Ue_1=\frac1{8\sqrt2}e_0+\frac1{32}e_1.
\]
Thus $E=\operatorname{span}\{e_0,e_1\}$ is invariant under $T_U$.  If
$f\in E^\perp$ and $e\in E$, then
$\langle T_Uf,e\rangle=\langle f,T_Ue\rangle=0$; hence $E^\perp$ is also
invariant.  Writing $A=T_U|_{E^\perp}$, we have
\[
 T_U=B_0\oplus A.
\]

By \Cref{lem:operator-identities}, $C_U$ represents $T_U^2$.  On the other
hand, \eqref{eq:low-block-square} and the definition of $L$ show that
$L(D_U(x),D_U(y))$ represents $B_0^2\oplus0$.  Therefore the kernel
\[
 K_U(x,y)=C_U(x,y)-L(D_U(x),D_U(y))
\]
represents $0_E\oplus A^2$, and thus $T_{K_U}$ is positive.  Equations
\eqref{eq:forced-codegree-identity} and \eqref{eq:certificate-polynomial}
now give
\[
 V_{\lambda^2}(D_U(x),D_U(y))K_U(x,y)
 =N_{\lambda^2}(D_U(x),D_U(y)),
 \qquad\text{for a.e. }(x,y)\in[0,1]^2.
\]

The function $D_U$ is bounded, and $K_U$ is a kernel.  Since
$\lambda^2\leq2^{-14}$, \Cref{lem:pole-exclusion} applies.  Its conclusion,
together with the definition of $R_z$, gives some
$0\leq b<a_{\lambda^2}$ such that $|D_U|\leq b$ almost everywhere and
\begin{equation}\label{eq:competitor-codegree}
 C_U(x,y)=R_{\lambda^2}(D_U(x),D_U(y))
 \qquad\text{for a.e. }(x,y)\in[0,1]^2.
\end{equation}
The series in \eqref{eq:codegree-series} converges uniformly and absolutely
on $[-b,b]^2$.

Put
\[
 m_n=\int_0^1P_n(D_U(x))\,\d x.
\]
Then $m_1=\int_0^1D_U(x)\,\d x=0$.  By the uniform absolute convergence
above, we may integrate \eqref{eq:competitor-codegree} term by term.
Using \Cref{lem:operator-identities}, \eqref{eq:target-path-density}, and
\eqref{eq:codegree-series}, we obtain
\[
\begin{aligned}
 \frac{33}{128}
 &=t(P_3,U)
 =\int_{[0,1]^2}C_U(x,y)\,\d x\,\d y\\
 &=\frac{33}{128}
   +\frac{17}{128}m_1
   +\frac9{512}m_1^2
   +\frac12\sum_{n=2}^{\infty}\lambda^{2n}m_n^2.
\end{aligned}
\]
Since $m_1=0$ and $\lambda>0$, it follows that
$m_n=0$ for every $n\geq1$.
Now we need the following lemma.

\begin{lem}\label[lem]{lem:chebyshev-coordinate}
Suppose that $D\colon[0,1]\to\R$ is bounded and measurable and
\[
    \int_0^1 P_n(D(x))\,\d x=0
    \qquad\text{for every }n\geq1.
\]
Then $D(x)\in[-1,1]$ for almost every $x\in[0,1]$.  After modifying $D$ on a null set, the
map
\[
    \varphi(x)=\frac{\arccos D(x)}{\pi}
\]
is measure preserving from $[0,1]$ to $[0,1]$, and
\[
    D(x)=\cos(\pi\varphi(x)),
    \qquad\text{a.e. }x\in[0,1].
\]
\end{lem}
\begin{proof}
Let $\mu$ and $\nu$ be the images of Lebesgue measure under $D$ and under
$u\mapsto\cos(\pi u)$, respectively.  For every $n\geq1$,
\[
 \int P_n\,\d\mu=0
 =\int_0^1\cos(n\pi u)\,\d u
 =\int P_n\,\d\nu,
\]
and the same equality holds for $n=0$ because $P_0=1$.
Since $\deg P_n=n$,
the Chebyshev polynomials span the polynomial algebra.
Hence $\mu$ and $\nu$ have the same integrals against every polynomial.
Both measures are supported on some compact interval $[-M,M]$, so the Stone-Weierstrass theorem implies 
$\mu$ and $\nu$ have the same integrals for every continuous function on $[-M,M]$. Thus $\mu=\nu$.

Since $\nu$ is supported on $[-1,1]$, we have $D(x)\in[-1,1]$ almost
everywhere.
Modify $D$ on the exceptional null set and put
$\varphi(x)=\arccos(D(x))/\pi$.
For $0\leq t\leq1$, the monotonicity of
$u\mapsto\cos(\pi u)$ gives
\[
 \operatorname{Leb}(\varphi^{-1}([0,t]))
 =\mu([\cos(\pi t),1])
 =\nu([\cos(\pi t),1])
 =t.
\]
The intervals $[0,t]$ determine Borel probability measures on $[0,1]$, so
$\varphi$ is measure preserving.
By construction, $D(x)=\cos(\pi\varphi(x))$ for every $x\in[0,1]$.
\end{proof}

Applying \Cref{lem:chebyshev-coordinate} to $D_U$, and modifying $D_U$ on a null set if necessary, we obtain a measure-preserving map
\[
    \varphi(x)=\frac{\arccos D_U(x)}{\pi}
\]
such that
\[
    D_U(x)=\cos(\pi\varphi(x)),
    \qquad\text{a.e. }x\in[0,1].
\]
Consequently, \eqref{eq:target-codegree} and
\eqref{eq:competitor-codegree} give
\[
\begin{aligned}
 C_U(x,y)
 &=R_{\lambda^2}(D_U(x),D_U(y))\\
 &=R_{\lambda^2}(\cos(\pi\varphi(x)),\cos(\pi\varphi(y)))\\
 &=C_{W_\lambda}(\varphi(x),\varphi(y)),
\end{aligned}
\qquad\text{a.e. }(x,y)\in[0,1]^2.
\]

Since $C_3\in\mathscr F$, we have
$t(C_3,U)=t(C_3,W_\lambda)$, and $T_{W_\lambda}$ is positive by \Cref{prop:target-properties}.
Hence \Cref{lem:positive-square-root} yields
\[
    U(x,y)=W_\lambda(\varphi(x),\varphi(y)),
    \qquad\text{a.e. }(x,y)\in[0,1]^2.
\]
Thus $U$ agrees almost everywhere with the pullback $W_\lambda^\varphi$.
Since $\varphi$ is measure preserving, $U$ is weakly isomorphic to $W_\lambda$.
This proves the forcing assertion for every bounded symmetric real-valued kernel.

Finally, note that
\[
    \frac{d}{d\lambda}
    \left(
      \frac{4505}{32768}
      +\frac{\lambda^6}{8(1-\lambda^3)}
    \right)
    =
    \frac{3\lambda^5(2-\lambda^3)}
         {8(1-\lambda^3)^2}
    >0.
\]
Thus \eqref{eq:target-triangle-density} distinguishes the weak-isomorphism classes of
$W_\lambda$ for different $\lambda\in (0,1/128]$.
\end{proof}

\section*{AI Declaration}

ChatGPT-5.6 sol, developed by OpenAI, was used during the early stages of this project and in the preparation of the manuscript.
In particular, the core idea underlying the proof of the main theorem was proposed by ChatGPT.
ChatGPT was also used to assist with the language, organization, and presentation of the manuscript.
The author subsequently developed and verified all mathematical arguments and takes full responsibility for all statements, proofs, citations, and conclusions presented in this paper.

\section*{Acknowledgements}

The author would like to thank Daniel Kr\'al' for the helpful discussion and suggestions on an earlier version of the manuscript.

\bibliographystyle{amsplain}
\bibliography{references}

@article{ChungGrahamWilson1989,
  author   = {Chung, Fan R. K. and Graham, Ronald L. and Wilson, Richard M.},
  title    = {Quasi-random graphs},
  journal  = {Combinatorica},
  volume   = {9},
  number   = {4},
  pages    = {345--362},
  year     = {1989},
  mrnumber = {1054011}
}

@article{LovaszSos2008,
  author   = {Lov{\'a}sz, L{\'a}szl{\'o} and S{\'o}s, Vera T.},
  title    = {Generalized quasirandom graphs},
  journal  = {J. Combin. Theory Ser. B},
  volume   = {98},
  number   = {1},
  pages    = {146--163},
  year     = {2008},
  mrnumber = {2368030}
}

@article{LovaszSzegedy2011,
  author   = {Lov{\'a}sz, L{\'a}szl{\'o} and Szegedy, Bal{\'a}zs},
  title    = {Finitely forcible graphons},
  journal  = {J. Combin. Theory Ser. B},
  volume   = {101},
  number   = {5},
  pages    = {269--301},
  year     = {2011},
  mrnumber = {2802882}
}

@article{GlebovKralVolec2019,
  author   = {Glebov, Roman and Kr{\'a}l', Daniel and Volec, Jan},
  title    = {Compactness and finite forcibility of graphons},
  journal  = {J. Eur. Math. Soc.},
  volume   = {21},
  number   = {10},
  pages    = {3199--3223},
  year     = {2019},
  mrnumber = {3994104}
}

@article{GlebovKlimosovaKral2019,
  author   = {Glebov, Roman and Klimo{\v{s}}ov{\'a}, Tereza and Kr{\'a}l', Daniel},
  title    = {Infinite-dimensional finitely forcible graphon},
  journal  = {Proc. Lond. Math. Soc. (3)},
  volume   = {118},
  number   = {4},
  pages    = {826--856},
  year     = {2019},
  mrnumber = {3938713}
}

@article{CooperKralMartins2018,
  author   = {Cooper, Jacob W. and Kr{\'a}l', Daniel and Martins, Ta{\'\i}sa L.},
  title    = {Finitely forcible graph limits are universal},
  journal  = {Adv. Math.},
  volume   = {340},
  pages    = {819--854},
  year     = {2018},
  mrnumber = {3886181}
}

@article{KralLovaszNoelSosnovec2020,
  author   = {Kr{\'a}l', Daniel and Lov{\'a}sz, L{\'a}szl{\'o} M. and Noel, Jonathan A. and Sosnovec, Jakub},
  title    = {Finitely forcible graphons with an almost arbitrary structure},
  journal  = {Discrete Anal.},
  pages    = {Paper No. 9, 36},
  year     = {2020},
  mrnumber = {4132060}
}

@book{Lovasz2012,
  author    = {Lov{\'a}sz, L{\'a}szl{\'o}},
  title     = {Large networks and graph limits},
  series    = {American Mathematical Society Colloquium Publications},
  volume    = {60},
  publisher = {American Mathematical Society},
  address   = {Providence, RI},
  pages     = {xiv+475},
  year      = {2012},
  isbn      = {978-0-8218-9085-1},
  mrnumber  = {3012035}
}

@book{Simon2005,
  author    = {Simon, Barry},
  title     = {Trace ideals and their applications},
  edition   = {Second},
  series    = {Mathematical Surveys and Monographs},
  volume    = {120},
  publisher = {American Mathematical Society,Providence, RI},
  pages     = {viii+150},
  year      = {2005},
  isbn      = {0-8218-3581-5},
  mrnumber  = {2154153}
}

@article{LovaszSzegedy2006,
  author   = {Lov{\'a}sz, L{\'a}szl{\'o} and Szegedy, Bal{\'a}zs},
  title    = {Limits of dense graph sequences},
  journal  = {J. Combin. Theory Ser. B},
  volume   = {96},
  number   = {6},
  pages    = {933--957},
  year     = {2006},
  mrnumber = {2274085}
}

@article{BorgsChayesLovaszSosVesztergombi2008,
  author   = {Borgs, Christian and Chayes, Jennifer T. and  Lov{\'a}sz, L{\'a}szl{\'o} and S{\'o}s, Vera T. and Vesztergombi, Katalin},
  title    = {Convergent sequences of dense graphs. I. Subgraph frequencies, metric properties and testing},
  journal  = {Adv. Math.},
  volume   = {219},
  number   = {6},
  pages    = {1801--1851},
  year     = {2008},
  mrnumber = {2455626}
}

@incollection{Thomason1987,
  author    = {Thomason, Andrew},
  title     = {Pseudorandom graphs},
  booktitle = {Random graphs '85 ({Pozna\'n}, 1985)},
  series    = {North-Holland Math. Stud.},
  volume    = {144},
  pages     = {307--331},
  publisher = {North-Holland,Amsterdam},
  year      = {1987},
  mrnumber  = {930498}
}

@article{CooperKaiserKralNoel2018,
  author   = {Cooper, Jacob W. and Kaiser, Tom{\'a}{\v{s}} and Kr{\'a}l', Daniel and Noel, Jonathan A.},
  title    = {Weak regularity and finitely forcible graph limits},
  journal  = {Trans. Amer. Math. Soc.},
  volume   = {370},
  number   = {6},
  pages    = {3833--3864},
  year     = {2018},
  mrnumber = {3811511}
}

@incollection{GrzesikKral2019,
  author    = {Grzesik, Andrzej and Kr{\'a}l', Daniel},
  title     = {Analytic representations of large graphs},
  booktitle = {Surveys in combinatorics 2019},
  series    = {London Math. Soc. Lecture Note Ser.},
  volume    = {456},
  pages     = {57--88},
  publisher = {Cambridge University Press},
  address   = {Cambridge},
  year      = {2019},
  mrnumber  = {3967292}
}

@article{GrzesikKralPikhurko2024,
  author   = {Grzesik, Andrzej and Kr{\'a}l', Daniel and Pikhurko, Oleg},
  title    = {Forcing generalised quasirandom graphs efficiently},
  journal  = {Combin. Probab. Comput.},
  volume   = {33},
  number   = {1},
  pages    = {16--31},
  year     = {2024},
  mrnumber = {4680487}
}

@book{MasonHandscomb2003,
  author    = {Mason, John C. and Handscomb, David C.},
  title     = {Chebyshev polynomials},
  publisher = {Chapman \& Hall/CRC},
  address   = {Boca Raton, FL},
  pages     = {xiii+341},
  year      = {2003},
  isbn      = {978-0-8493-0355-5},
  mrnumber  = {1937591}
}

@article{Szablowski2019,
  author   = {Szab\l owski, Pawe\l J.},
  title    = {On probabilistic aspects of Chebyshev polynomials},
  journal  = {Statist. Probab. Lett.},
  fjournal = {Statistics \& Probability Letters},
  volume   = {145},
  year     = {2019},
  pages    = {205--215},
  issn     = {0167-7152,1879-2103},
  mrclass  = {60E05 (41A50 42A16)},
  mrnumber = {3873908},
}

\bigskip
\noindent
\textit{Junchi Zhang, Shanghai Center for Mathematical Sciences,
Fudan University, Shanghai, China}\\
\textit{Email:} \texttt{jczhang24@m.fudan.edu.cn}

\end{document}